\documentclass[a4paper, 12pt]{article}

\usepackage{amsfonts, amsmath, amsthm}
\usepackage{amssymb}
\usepackage{latexsym}
\usepackage[dvips]{graphicx}
\usepackage{color}

\theoremstyle{plain}
\newtheorem{thm}{Theorem}[section]
\newtheorem{prop}{Proposition}[section]
\newtheorem{lem}{Lemma}[section]
\newtheorem{cor}{Corollary}[section]
\theoremstyle{definition}

\theoremstyle{remark}

\numberwithin{equation}{section}

\newcommand{\R}{\mathbb{R}}
\newcommand{\C}{\mathbb{C}}

\newcommand{\cc}[1]{\overline{#1}}
\newcommand{\op}[1]{\mathcal{#1}}

\newcommand{\pa}{\partial}
\newcommand{\eps}{\varepsilon}
\newcommand{\jb}[1]{\langle #1 \rangle}

\DeclareMathOperator{\realpart}{\rm Re}

\newcommand{\dis}{\displaystyle}

\begin{document}
\title{
Large time asymptotics for a cubic nonlinear Schr\"odinger system
in one space dimension, II\\ 
 }

\author{
          Chunhua Li  \thanks{
              Department of Mathematics, College of Science,
              Yanbian University.
              977 Gongyuan Road, Yanji,
              Jilin 133002, China.
              (E-mail: {\tt sxlch@ybu.edu.cn})
             }
          \and
          Yoshinori Nishii\thanks{
              Department of Mathematics, Graduate School of Science,
              Osaka University.
              1-1 Machikaneyama-cho, Toyonaka,
              Osaka 560-0043, Japan.
              (E-mail: {\tt y-nishii@cr.math.sci.osaka-u.ac.jp})             }
           \and
          Yuji Sagawa \thanks{
             Micron Memory Japan, G.K.
             7-10 Yoshikawakogyodanchi, Higashihiroshima,
             Hiroshima 739-0153, Japan.}
           \and
          Hideaki Sunagawa \thanks{
              Department of Mathematics, Graduate School of Science,
              Osaka University.
              1-1 Machikaneyama-cho, Toyonaka,
              Osaka 560-0043, Japan.
              (E-mail: {\tt sunagawa@math.sci.osaka-u.ac.jp})
             }
}

\date{\today }
\maketitle

\noindent{\bf Abstract:}\ This is a sequel to the paper
``Large time asymptotics for a cubic nonlinear Schr\"odinger system
in one space dimension" by the same authors. 
We continue to study the Cauchy problem for the two-component system of cubic 
nonlinear Schr\"odinger equations in one space dimension. 
We provide criteria for
large time decay or non-decay in $L^2$ of the small amplitude solutions in
terms of the Fourier transforms of the initial data.
\\

\\


\section{Introduction}  \label{sec_intro}
This is a sequel  to the paper \cite{LNSS}.
We continue to study the Cauchy problem for
\begin{align}
\left\{\begin{array}{ll}
\begin{array}{l}
 \op{L} u_1=- i|u_2|^2 u_1,\\
 \op{L} u_2=- i|u_1|^2 u_2,
 \end{array}
 & (t,x )\in (0,\infty)\times \R,
 \end{array}\right.
\label{eq}
\end{align}
with
\begin{align}
u_j(0,x)=\varphi^0_j(x),
\qquad  x \in \R,\ j=1,2,
\label{data}
\end{align}
where $i=\sqrt{-1}$, $\op{L}=i\pa_t+(1/2)\pa_{x}^2$, and
$\varphi^0=(\varphi^0_1(x), \varphi_2^0(x))$ is a given $\C^2$-valued
function of $x\in \R$.
The following result has been obtained in \cite{LNSS}.
As in \cite{LNSS}, $H^{s,\sigma}$ stands for the $L^2$-based weighted
Sobolev space of order $s$, $\sigma$. We also write $H^s$ for $H^{s,0}$.
\begin{prop}[\cite{LNSS}]\label{prop_lnss1}
Suppose that $\varphi^0=(\varphi_1^0,\varphi_2^0)\in H^{2}\cap H^{1,1}$
and $\|\varphi^0\|_{H^2\cap H^{1,1}}$ is suitably small.
Let $u=(u_1,u_2)\in C([0,\infty); H^2\cap H^{1,1})$
be the solution to \eqref{eq}--\eqref{data}. 
Then there exists $\varphi^+=(\varphi_1^+,\varphi_2^+)\in L^2$ with
$\hat{\varphi}^+=(\hat{\varphi}_1^+,\hat{\varphi}_2^+)\in L^{\infty}$
such that
\begin{align}\label{scattering}
 \lim_{t\to +\infty} \|u_j(t)-\op{U}(t)\varphi_j^+\|_{L^2}=0,\quad j=1,2,
\end{align}
where $\op{U}(t)=\exp(i\frac{t}{2}\pa_x^2)$.
Moreover we have
\begin{align}\label{comp_rel}
\hat{\varphi}_1^+(\xi)\cdot \hat{\varphi}_2^+(\xi)=0
\end{align}
for each $\xi \in \R$, where 
$\hat{\phi}=\op{F}\phi$ denotes the Fourier transform of a function $\phi$.
\end{prop}

As we have mentioned in \cite{LNSS}, the most important thing in
Proposition~\ref{prop_lnss1} is
the relation \eqref{comp_rel} which does not appear in the
usual short-range situation.
The purpose of this sequel is to investigate it in more detail.
For this purpose, let us recall the following proposition also shown
in \cite{LNSS}.

\begin{prop}[\cite{LNSS}]\label{prop_key_lnss1}
We put $\dis{\varphi_j^+=\lim_{t\to+\infty}\op{U}(-t)u_j(t)}$
in $L^2$, $j=1,2$, for the global solution $u=(u_1,u_2)$ to
\eqref{eq}--\eqref{data},
whose existence is guaranteed by Proposition~\ref{prop_lnss1}.
There exists a function $m:\R\to\R$ such that the following holds
for each $\xi \in \R$
\begin{itemize}
\item
$m(\xi)>0$ implies
$\hat{\varphi}_1^+(\xi)\ne 0$ and $\hat{\varphi}_2^+(\xi) = 0$;
\item
$m(\xi)<0$ implies
$\hat{\varphi}_1^+(\xi)= 0$ and $\hat{\varphi}_2^+(\xi)\ne 0$;
\item
$m(\xi)=0$ implies
$\hat{\varphi}_1^+(\xi)=\hat{\varphi}_2^+(\xi)=0$.
\end{itemize}
\end{prop}

Note that \eqref{comp_rel} is an immediate consequence of
Proposition~\ref{prop_key_lnss1}. In other words,
Proposition~\ref{prop_key_lnss1} is more precise than
\eqref{comp_rel}, and the function $m(\xi)$ plays an important role
in it. This indicates that better understanding of $m(\xi)$ 
will bring us more precise information on the scattering state 
$\varphi^+$. 

Our aim in the present paper is to specify the leading term of
$m(\xi)$ for sufficiently small initial data. This will allow us to
find criteria for $L^2$ decay/non-decay of each component of the
solutions to \eqref{eq}. 

\section{The leading term of $m(\xi)$ in the small amplitude limit}
\label{sec_asymp_m}
In what follows, we put a small parameter $\eps$ in front of
the initial data to distinguish information on the amplitude from the others,
that is, we replace the initial condition \eqref{data} by
\begin{align}
 u_j(0,x) =\eps \psi_j(x),\quad j=1,2,
\label{data_eps}
\end{align}
where $\psi_j\in H^{2}\cap H^{1,1}$ is independent of $\eps$.
Our main theorem reads as follows.

\begin{thm}\label{thm_main}
Let $m$ be the function given in Proposition~\ref{prop_key_lnss1}
with the initial condition \eqref{data} replaced by \eqref{data_eps}.
We have
\begin{align*}
m(\xi)
=
\eps^2 \bigl(|\hat{\psi}_1(\xi)|^2-|\hat{\psi}_2(\xi)|^2\bigr)
+
O(\eps^4)
\end{align*}
as $\eps \to +0$ uniformly in $\xi \in \R$.
\end{thm}

As a consequence of Theorem~\ref{thm_main}, we have the following criteria 
for (non-)triviality of the scattering state
$\varphi^+=(\varphi_1^+,\varphi_2^+)$ for \eqref{eq}--\eqref{data_eps}.

\begin{cor}\label{cor_criterion1}
Assume that there exist points $\xi^* \in \R$ and $\xi_* \in \R$ such that
\begin{align}
 |\hat{\psi}_1(\xi^*)| > |\hat{\psi}_2(\xi^*)|
\label{m_plus}
\end{align}
and
\begin{align}
 |\hat{\psi}_1(\xi_*)| < |\hat{\psi}_2(\xi_*)|,
\label{m_minus}
\end{align}
respectively.
Then, for sufficiently small $\eps$, we have
$\|\varphi_1^+\|_{L^2}>0$ and $\|\varphi_2^+\|_{L^2}>0$.
\end{cor}

\begin{cor}\label{cor_criterion2}
Assume that
\begin{align}
 |\hat{\psi}_1(\xi)| > |\hat{\psi}_2(\xi)|
\label{m_plus_everywhere}
\end{align}
for all $\xi \in \R$. Then, for sufficiently small $\eps$,
$\varphi_2^+$ vanishes almost everywhere on $\R$,
while $\|\varphi_1^+\|_{L^2}>0$.
\end{cor}

It follows from  \eqref{scattering} and Corollary~\ref{cor_criterion1}
that both $u_1(t)$ and $u_2(t)$ behave like non-trivial free solutions as
$t\to +\infty$. In particular, we see that
$L^2$ decay does not occur for $u_1(t)$ and $u_2(t)$
under \eqref{m_plus} and \eqref{m_minus}.
To the contrary,  Corollary~\ref{cor_criterion2} tells us
that only the second component $u_2(t)$ is dissipated as
$t\to \infty$ in the sense of $L^2$ under \eqref{m_plus_everywhere}.
We emphasize again that such phenomena do not occur in the usual short-range
settings. In this sense, the dynamics for the system \eqref{eq} is much more
delicate than that for the single Schr\"odinger equation with 
dissipative cubic nonlinear terms.

\section{Proofs}
This section is devoted to the proofs of Theorem~\ref{thm_main} and its
corollaries.
In what follows, we will denote various positive constants by the same letter
$C$, which may vary from one line to another.

\subsection{Proof of Theorem~\ref{thm_main}}
We set
$\dis{\alpha_j(t,\xi)=\op{F} \bigl[\op{U}(-t)u_j(t,\cdot)\bigr](\xi)}$
for the solution $u=(u_1,u_2)$ to \eqref{eq}--\eqref{data_eps}.
According to \cite{LNSS}, we have in fact the following expression for $m(\xi)$ 
in Proposition~\ref{prop_key_lnss1}:
\begin{align}\label{m}
m(\xi)=\left|\alpha_1(2,\xi)\right|^2-\left|\alpha_2(2,\xi)\right|^2+\int_2^{\infty}\rho(\tau,
\xi)d\tau,
\end{align}
where
\[
 \rho(t,\xi)
 =
 2\realpart\Bigl[ 
  \cc{\alpha_1(t,\xi)}R_1(t,\xi)- \cc{\alpha_2(t,\xi)}R_2(t,\xi) 
  \Bigr],
\]
\[
R_1
=
\frac{1}{t} |\alpha_2|^2 \alpha_1
-\op{F}\op{U}(-t)\bigl[ |u_2|^2 u_1\bigr],
\]
\[
R_2
=
\frac{1}{t} |\alpha_1|^2 \alpha_2
-\op{F}\op{U}(-t)\bigl[ |u_1|^2 u_2\bigr].
\]
From the argument of Section~3 in \cite{LNSS}, we already know
the following  estimate for $\rho$: 
\begin{align}\label{rho}
\int_2^{\infty}|\rho(\tau,\xi)|d\tau
\le
 C\eps^4 \jb{\xi}^{-2}.
\end{align}
It follows from \eqref{m} and \eqref{rho} that
\begin{align*}
 \sup_{\xi\in \R}
 \Bigl|
 m(\xi)
 -
 \bigl(|\alpha_1(2,\xi)|^2-|\alpha_2(2,\xi)|^2\bigr)
 \Bigr|
 \le
 C\eps^4.
\end{align*}
Therefore, it suffices to prove the following lemma.
\begin{lem}\label{lem_short_time} For $j=1$, $2$, we have
\[
\alpha_j (2,\xi) =\eps \hat{\psi}_j(\xi)+ O(\eps^{3})
\]
as $\eps \to +0$, uniformly in $\xi \in \R$.
\end{lem}
\begin{proof}
First we recall the estimates which have been shown in 
\cite{LiS} (or \cite{Kim}):
\begin{align*}
 \|u(t)\|_{L^2}+\|\op{J}u(t)\|_{L^2} \le C\eps (1+t)^{\gamma},
\end{align*}
\begin{align*}
 \|u(t)\|_{L^{\infty}} \le C\eps (1+t)^{-1/2},
\end{align*}
where $\gamma\in (0,1/12)$ and $\op{J}=x+it\pa_x$.
We set $N_1(u)=|u_2|^2u_1$ and $N_2(u)=|u_1|^2u_2$.
Then it follows from the relations
$\partial_{t}\op{U}(-t)u_j=-\op{U}(-t)N_j(u)$, 
$\op{U}(t)x\op{U}(-t)=\op{J}$ and the Sobolev embedding that
\begin{align*}
 \sup_{\xi\in \R}\bigl|\alpha_j (2,\xi) -\eps \hat{\psi}_j(\xi)\bigr|
&\le
C\bigl\|\op{U}(-2)u_j (2,\cdot) -u_j(0,\cdot)\bigr\|_{H^{0,1}}\\
&\le
C\int_0^2 \|\op{U}(-\tau) N_j(u(\tau))\|_{H^{0,1}}\, d\tau\\
&\le
C\int_0^2 \|u(\tau)\|_{L^{\infty}}^2
\bigl(\|u(\tau)\|_{L^2}+\|\op{J}u(\tau)\|_{L^2}\bigr)\, d\tau\\
&\le
C\eps^3.
\end{align*}
\end{proof}

\subsection{Proof of Corollary~\ref{cor_criterion1}}
We put $V=\{\xi \in \R\, |\, |\hat{\psi}_1(\xi)| > |\hat{\psi}_2(\xi)|\}$.
By \eqref{m_plus}, we see that $V$ is a non-empty open set.
Now we take $r>0$ so small  that the closed interval
$I=[\xi^*-r, \xi^*+r]$ is included in $V$, and we put
\[
 C_1
=
\min_{\xi \in I}
\bigl(|\hat{\psi}_1(\xi)|^2-|\hat{\psi}_2(\xi)|^2\bigr).
\]
Then we have $C_1>0$, and Theorem~\ref{thm_main} gives us
\[
 m(\xi) \ge C_1\eps^2 -C\eps^4>0
\]
for $\xi \in I$, if $\eps>0$ is small enough.
By Proposition~\ref{prop_key_lnss1}, we have
$\hat{\varphi}_1^+(\xi)\ne 0$ for $\xi \in I$. Therefore we obtain
\[
\|\varphi_1^+\|_{L^2}
\ge
\|\hat{\varphi}_1^+\|_{L^2(I)}
>0.
\]
Similarly, \eqref{m_minus} yields $\|{\varphi}_2^+\|_{L^2}>0$.
\qed\\

\subsection{Proof of Corollary~\ref{cor_criterion2}}
Let $\chi:\R\to \R$ be a cut-off function satisfying
$\chi(\xi)=1$ ($|\xi|\le 1$) and $\chi(\xi)=0$ ($|\xi|\ge 2$).
For given $\delta>0$, we can choose $q \ge 1$ so large that
$\|(1-\chi_{q})\hat{\varphi}_2^+\|_{L^2} < \delta$, 
where $\chi_{q}(\xi)=\chi(\xi/q)$. With this $q$, we put
\[
 C_2
=
\min_{|\xi|\le 2q} 
\bigl(|\hat{\psi}_1(\xi)|^2-|\hat{\psi}_2(\xi)|^2\bigr). 
\]
Then we have $C_2>0$, because of \eqref{m_plus_everywhere}.
So it follows from Theorem~\ref{thm_main} that
\[
 m(\xi) \ge C_2\eps^2 -C\eps^4>0
\]
for $|\xi|\le 2q$, if $\eps>0$ is small enough. By
Proposition~\ref{prop_key_lnss1}, we deduce that
$\chi_{q}(\xi)\hat{\varphi}_2^+(\xi)=0$ for all $\xi \in \R$.
Therefore
\[
 \|\varphi_2^+\|_{L^2} = \|(1-\chi_{q})\hat{\varphi}_2^+\|_{L^2} <\delta.
\]
Since $\delta$ can be taken arbitrarily small,
this means that $\varphi_2^+$ vanishes
almost everywhere on $\R$.
\qed

\medskip
\subsection*{Acknowledgments}
The authors would like to express their gratitude to Professors
Soichiro Katayama, Naoyasu Kita and Satoshi Masaki for
their valuable comments on the paper \cite{LNSS}, which motivate the
present work.
This work is supported by the Research Institute for Mathematical 
Sciences, an International Joint Usage/Research Center located in Kyoto 
University. 
The work of H.~S. is supported by Grant-in-Aid for Scientific Research (C)
(No.~17K05322), JSPS.


\end{document}